\documentclass[review]{elsarticle}
\usepackage{CJK}
\usepackage[usenames]{color}
\usepackage{extarrows}
\usepackage{lineno,hyperref}
\modulolinenumbers[5]
\journal{Journal of Filomat}
\usepackage{amsfonts}
\usepackage{amsmath, cases}
\usepackage{indentfirst}
\usepackage{graphicx}
\usepackage{latexsym,bm, amsthm}

\usepackage{epsfig}
\usepackage{latexsym}
\usepackage{amssymb}

\begin{document}
\begin{CJK*}{GBK}{song}

\begin{frontmatter}

\title{
Three limit representations of the core-EP inverse 
}
\author[1-address]{Mengmeng Zhou}
 \ead{mmz9209@163.com)}
\address[1-address]{School of Mathematics, Southeast University, Nanjing, Jiangsu 210096, China}

\author[2-address]{Jianlong Chen\corref{mycorrespondingauthor}}
\cortext[mycorrespondingauthor]{Corresponding author}
\ead{jlchen@seu.edu.cn)}
\address[2-address]{School of Mathematics, Southeast University, Nanjing, Jiangsu, 210096, China}

\author[3-address]{Tingting Li}
 \ead{littnanjing@163.com)}
\address[3-address]{School of Mathematics, Southeast University, Nanjing, Jiangsu 210096, China}

\author[4-address]{Dingguo Wang}
\ead{dingguo95@126.com}
\address[4-address]{School of Mathematical Sciences, Qufu Normal University,\\ Qufu, Shandong, 273165, China}

\begin{abstract}
 ~~~In this paper, we present three limit representations of the core-EP inverse. The first approach is based on the full-rank decomposition of a given matrix. The second and third approaches, which depend on the explicit expression of the core-EP inverse, are established. The corresponding limit representations of the dual core-EP inverse are also given. In particular, limit representations of the core and dual core inverse are derived.
\end{abstract}


\end{frontmatter}

\linenumbers

\section{Introduction}
\newcommand{\song}{\CJKfamily{song}}    
\newcommand{\kai}{\CJKfamily{kai}}      
\newcommand{\hei}{\CJKfamily{hei}}      
\newcommand{\you}{\CJKfamily{you}}      
\newcommand{\fs}{\CJKfamily{fs}}        
\numberwithin{equation}{section} \theoremstyle{plain}
\newtheorem*{thm*}{Main Theorem}
\newtheorem{theorem}{Theorem}[section]
\newtheorem{corollary}[theorem]{Corollary}
\newtheorem*{corollary*}{Corollary}
\newtheorem{claim}[theorem]{Claim}
\newtheorem*{claim*}{Claim}
\newtheorem{lemma}[theorem]{Lemma}
\newtheorem*{lemma*}{Lemma}
\newtheorem{proposition}[theorem]{Proposition}
\newtheorem*{proposition*}{Proposition}
\newtheorem{remark}[theorem]{Remark}
\newtheorem*{remark*}{Remark}
\newtheorem{example}[theorem]{Example}
\newtheorem*{example*}{Example}
\newtheorem{question}[theorem]{Question}
\newtheorem*{question*}{Question}
\newtheorem{definition}[theorem]{Definition}
\newtheorem*{definition*}{Definition}
\newtheorem{conjecture}{Conjecture}

In 1974, limit representation of the Drazin inverse was established by Meyer \cite{M}. In 1986, Kalaba and Rasakhoo \cite{KR} introduced limit representation of the Moore-Penrose inverse. In 1994, An alternative limit representation of the Drazin inverse was given by Ji \cite{J}. It is well known that the six kinds of classical generalized inverses: the Moore-Penrose inverse, the weighted Moore-Penrose inverse, the group inverse, the Drazin inverse, the Bott-Duffin inverse and the generalized Bott-Duffin inverse can be presented as particular generalized inverse $A^{(2)}_{T,S}$ with prescribed range and null space (see, for example, \cite{BG, CM, YCH, CSK}).

Let $T$ be a subspace of $\mathbb{C}^{n}$ and let $S$ be a subspace of $\mathbb{C}^{m}.$ The generalized inverse $A_{T,S}^{(2)}$ \cite{BG} of matrix $A\in\mathbb{C}^{m\times n}$ is the matrix $G\in\mathbb{C}^{n\times m}$ satisfying
$$ GAG=G, \mathcal{R}(G)=T, \mathcal{N}(G)=S.$$
In 1998, Wei \cite{W} established a unified limit representation of the generalized inverse $A^{(2)}_{T,S}.$ Let $A\in \mathbb{C}^{m\times n}$ be matrix of rank $r,$ let $T$ be a subspace of $\mathbb{C}^{n}$ of dimension $s\leq r$, and let $S$ be a subspace of $\mathbb{C}^{m}$ of dimension $m-s.$ Suppose $G\in \mathbb{C}^{n\times m}$ such that $R(G)=T$ and $N(G)=S.$ If $A^{(2)}_{T,S}$ exists, then
$$A^{(2)}_{T,S}=\lim_{\lambda\rightarrow 0}(GA-\lambda I)^{-1}G=\lim_{\lambda\rightarrow 0}G(AG-\lambda I)^{-1},$$
where $R(G)$ and $N(G)$ denote the range space and null space of $G,$ respectively.

In 1999, Stanimirovi\'{c} \cite{S} introduced a more general limit formula. Let $M$ and $N$ are two arbitrary $p\times q$ matrices, then
\begin{equation}\label{a1}
 \lim_{\lambda\rightarrow 0}(M^{*}N+\lambda I)^{-1}M^{*}=\lim_{\lambda\rightarrow 0}M^{*}(NM^{*}+\lambda I)^{-1}.
\end{equation}
The limit representation of generalized inverse $A^{(2)}_{T,S}$ is a special case of the above general formula.

In 2012, Liu et al. \cite{LYZW} introduced limit representation of the generalized inverse $A^{(2)}_{R(B),N(C)}$ in Banach space. Let $A\in \mathbb{C}^{m\times n}$ be matrix of rank $r$. Let $T$ be a subspace of $\mathbb{C}^{n}$ of dimension $s\leq r$, and let $S$ be a subspace of $\mathbb{C}^{m}$ of dimension $m-s.$ Suppose $B\in\mathbb{C}^{n\times s}$ and $C\in\mathbb{C}^{s\times m}$ such that $R(B)=T$ and $N(C)=S.$ If $A^{(2)}_{T,S}$ exists, then
$$A^{(2)}_{T,S}=\lim_{\lambda\rightarrow 0}B(CAB+\lambda I)^{-1}C. $$

In 2010, the core and dual core inverse were introduced by Baksalary and Trenkler for square matrices of index at most 1 in \cite{BT}. In 2014, the core inverse was extended to the core-EP inverse defined by Manjunatha Prasad and Mohana \cite{MPM}. The core-EP inverse coincides with the core inverse if the index of a given matrix is 1. In this paper, the dual conception of the core-EP inverse was called the dual core-EP inverse. The characterizations of the core-EP and core inverse were investigated in complex matrices and rings (see, for example, \cite{ZPC, YC, LC, RDD, HXW, WL, XCZ}).

From the above mentioned limit representations of the generalized inverse, we know that limit representation of the core-EP inverse similar to the form of (\ref{a1}) has not been investigated in the literature.

The purpose of this paper is to establish three limit representations of the core-EP inverse. The first approach is based on the full-rank decomposition of a given matrix. The second and third approaches, which depend on the explicit expression of the core-EP inverse, are represented. The corresponding limit representations of the dual core-EP inverse are also given. In particular, limit representations of the core and dual core inverse are derived.

\section{Preliminaries}
\label{Preliminaries}
In this section, we give some auxiliary definitions and lemmas.

For arbitrary matrix $A\in\mathbb{C}^{m\times n},$ the symbol $\mathbb{C}^{m\times n}$ denotes the set of all complex $m\times n$ matrices. $A^{*}$ and ${\rm rk}(A)$ denote the conjugate transpose and rank of $A$, respectively. $I$ is the identity matrix of an appropriate order. If $k$ is the smallest nonnegative integer such that ${\rm rk}(A^{k})={\rm rk}(A^{k+1}),$ then $k$ is called the index of $A$ and denoted by ${\rm ind}(A)$.

\begin{definition}\label{MP} \emph{\cite{BG}} Let $A\in\mathbb{C}^{m\times n}.$  The unique matrix $A^{\dagger}\in\mathbb{C}^{n\times m}$ is called the Moore-Penrose inverse of $A$ if it satisfies
$$AA^{\dagger}A=A, ~A^{\dagger}AA^{\dagger}=A^{\dagger}, ~(AA^{\dagger})^{*}=AA^{\dagger}, ~(A^{\dagger}A)^{*}=A^{\dagger}A.$$
\end{definition}

\begin{definition}\label{DI} \emph{\cite{BG}} Let $A\in\mathbb{C}^{n\times n}.$ The unique matrix $A^{D}\in\mathbb{C}^{n\times n}$ is called the Drazin inverse of $A$ if it satisfies
 $$ A^{k+1}A^{D}=A^{k},  ~A^{D}AA^{D}=A^{D},  ~AA^{D}=A^{D}A,$$
where $k={\rm ind}(A).$  When $k=1,$ the Drazin inverse reduced to the group inverse and it is denoted by $A^\#.$
\end{definition}

\begin{definition}\label{CD} \emph{\cite{BT}} A matrix $A^{\tiny\textcircled{\tiny\#}}\in\mathbb{C}^{n\times n}$ is called the core inverse of $A\in\mathbb{C}^{n\times n}$ if it satisfies
$$AA^{\tiny\textcircled{\tiny\#}}=P_{A} ~{\rm and} ~R(A^{\tiny\textcircled{\tiny\#}})\subseteq R(A) .$$
Dually, A matrix $A_{\tiny\textcircled{\tiny\#}}\in\mathbb{C}^{n\times n}$ is called the dual core inverse of $A\in\mathbb{C}^{n\times n}$ if it satisfies $$A_{\tiny\textcircled{\tiny\#}}A=P_{A^{*}} ~{\rm and} ~R(A_{\tiny\textcircled{\tiny\#}})\subseteq R(A^{*}) .$$
\end{definition}

\begin{definition}\label{cp} \emph{\cite{MPM}} A matrix $X\in\mathbb{C}^{n\times n}$, denoted by $A^{\tiny\textcircled{\tiny$\dagger$}}$, is called the core-EP inverse of $A\in\mathbb{C}^{n\times n}$ if it satisfies
$$XAX=X~and~ R(X)=R(X^{*})=R(A^{D}).$$
Dually, A matrix $X\in\mathbb{C}^{n\times n}$, denoted by $A_{\tiny\textcircled{\tiny$\dagger$}}$, is called the dual core-EP inverse of $A\in\mathbb{C}^{n\times n}$ if it satisfies
$$XAX=X~and~ R(X)=R(X^{*})=R((A^{*})^{D}).$$
\end{definition}

The core-EP inverse was extended from complex matrices to rings by Gao and Chen in \cite{YC}.
\begin{lemma}\label{db4} \emph{\cite{YC}} Let $A\in\mathbb{C}^{n\times n}$ with ${\rm ind}(A)=k$ and let $m$ be a positive integer with $m\geq k.$ Then
$ A^{\tiny\textcircled{\tiny$\dagger$}}=A^{D}A^{m}(A^{m})^{\dagger}.$
\end{lemma}
Clearly, If $A\in\mathbb{C}^{n\times n}$ with ${\rm ind}(A)=k,$ then it has a unique core-EP inverse. So according to Lemma \ref{db4}, we have
\begin{equation}\label{c16}
 A^{\tiny\textcircled{\tiny$\dagger$}}=A^{D}A^{k}(A^{k})^{\dagger}
  =A^{D}A^{k+1}(A^{k+1})^{\dagger}
  =A^{k}(A^{k+1})^{\dagger}.
\end{equation}

As for an arbitrary matrix $A\in\mathbb{C}^{n\times n}$. If $A$ is nilpotent, then $A^{D}=0.$ In this case, $A^{\tiny\textcircled{\tiny$\dagger$}}
=A_{\tiny\textcircled{\tiny$\dagger$}}=0.$ This case is considered to be trivial. So we restrict the matrix $A$ to be non-nilpotent in this paper.

\begin{lemma}\label{db1}\emph{\cite{BG}}  Let $A\in\mathbb{C}^{n\times n}$ with ${\rm ind}(A)=k.$ If $A=B_{1}G_{1}$ is a full-rank decomposition and $G_{i}B_{i}=B_{i+1}G_{i+1}$ are also full-rank decompositions, $i=1,2, ...,k-1. $ Then the following statements hold:
\begin{itemize}
\item[$(1)$] $G_{k}B_{k}$ is invertible.
\item[$(2)$] $A^{k}=B_{1}B_{2}\cdot\cdot\cdot B_{k}G_{k}\cdot\cdot\cdot G_{2}G_{1}.$
\item[$(3)$] $A^{D}=B_{1}B_{2}\cdot\cdot\cdot B_{k}(G_{k}B_{k})^{-k-1}G_{k}\cdot\cdot\cdot G_{1}.$
\end{itemize}
In particular, for $k=1,$ then $G_{1}B_{1}$ is invertible and
$$A^\#=B_{1}(G_{1}B_{1})^{-2}G_{1}.$$
\end{lemma}

According to \cite{BG}, it is also known that $A^{\dagger}=G_{1}^{*}(G_{1}G_{1}^{*})^{-1}(B_{1}^{*}B_{1})^{-1}B_{1}^{*}.$
\begin{lemma}\label{db2} \emph{\cite{WL}} Let $A\in\mathbb{C}^{n\times n}$ with ${\rm ind}(A)=1.$ If $A=MN$ is a full-rank decomposition, then
$$A^{\tiny\textcircled{\tiny\#}}=M(NM)^{-1}(M^{*}M)^{-1}M^{*}.$$
\end{lemma}

\begin{lemma}\label{db3} \emph{\cite{HXW}} Let $A\in\mathbb{C}^{n\times n}$ with ${\rm ind}(A)=k.$ Then $A$ can be written as the sum of matrices $A_{1}$ and $A_{2}$, i.e., $A=A_{1}+A_{2},$ where
\begin{itemize}
\item[$(1)$] ${\rm rk}(A^{2}_{1})={\rm rk}(A_{1})$.
\item[$(2)$] $A^{k}_{2}=0.$
\item[$(3)$] $A^{*}_{1}A_{2}=A_{2}A_{1}=0.$
\end{itemize}
\end{lemma}

\section{The first approach }
In this section, we present limit representations of the core-EP and dual core-EP inverse, which depend on the full-rank decomposition of a given matrix. In particular, limit representations of the core and dual core inverse are also given.

\begin{theorem}\label{e2}  Let $A\in\mathbb{C}^{n\times n}$ with ${\rm ind}(A)=k,$ and the full-rank decomposition of $A$ be as in Lemma \ref{db1}. Then
\begin{eqnarray*}
 A^{\tiny\textcircled{\tiny$\dagger$}}&=&\lim_{\lambda\rightarrow 0 }B(BG_{k}B_{k})^{*}(BG_{k}B_{k}(BG_{k}B_{k})^{*}+\lambda I)^{-1}\\
   &=& \lim_{\lambda\rightarrow 0 }B((BG_{k}B_{k})^{*}BG_{k}B_{k}+\lambda I)^{-1}(BG_{k}B_{k})^{*}
\end{eqnarray*}
and
\begin{eqnarray*}
  A_{\tiny\textcircled{\tiny$\dagger$}} &=&\lim_{\lambda\rightarrow 0 }((G_{k}B_{k}G)^{*}G_{k}B_{k}G+\lambda I)^{-1}(G_{k}B_{k}G)^{*}G \\
   &=&\lim_{\lambda\rightarrow 0 }(G_{k}B_{k}G)^{*}(G_{k}B_{k}G(G_{k}B_{k}G)^{*}+\lambda I)^{-1}G,
\end{eqnarray*}
where $B=B_{1}B_{2}\cdot\cdot\cdot B_{k}$ and $G=G_{k}\cdot\cdot\cdot G_{2}G_{1}.$
\end{theorem}

\begin{proof} Let $B=B_{1}B_{2}\cdot\cdot\cdot B_{k}$, $G=G_{k}\cdot\cdot\cdot G_{2}G_{1}$ and $X=B(BG_{k}B_{k})^{\dagger}.$ Suppose that $A\in\mathbb{C}^{n\times n}$ with ${\rm ind}(A)=k,$ and the full-rank decomposition of $A$ be as in Lemma \ref{db1}. From \cite{S}, we know that
\begin{equation}\label{h1}
 A^{\dagger}=\lim_{\lambda\rightarrow 0 }A^{*}(AA^{*}+\lambda I)^{-1}=\lim_{\lambda\rightarrow 0 }(A^{*}A+\lambda I)^{-1}A^{*}.
\end{equation}
So it is sufficient to verify $X= A^{\tiny\textcircled{\tiny$\dagger$}}.$ Since
\begin{equation}\label{e3}
  AB = B_{1}G_{1}B_{1}\cdot\cdot\cdot B_{k}
   =  B_{1}B_{2}G_{2}B_{2}\cdot\cdot\cdot B_{k}
   = \cdot\cdot\cdot
   =  BG_{k}B_{k},
\end{equation}
$$XAX=B(BG_{k}B_{k})^{\dagger}AB(BG_{k}B_{k})^{\dagger}\overset{(\ref{e3})}{=}B(BG_{k}B_{k})^{\dagger}=X.$$
From Lemma \ref{db1},we obtain
$$B(BG_{k}B_{k})^{\dagger}=B(G_{k}B_{k})^{-1}(B^{*}B)^{-1}B^{*}.$$
Therefore, $$X^{*}=B(B^{*}B)^{-1}((G_{k}B_{k})^{-1})^{*}B^{*}.$$
It is easy verify that $B^{*}B,$ $GG^{*}$ and $G_{k}B_{k}$ are invertible. Hence
$${\rm rk}(B)={\rm rk}(B(G_{k}B_{k})^{*}B^{*}B)\leq {\rm rk}(B(BG_{k}B_{k})^{*})\leq rk(B),$$
$${\rm rk}(B)={\rm rk}(B(B^{*}B)^{-1}((G_{k}B_{k})^{-1})^{*}B^{*}B)\leq {\rm rk}(X^{*})\leq rk(B),$$
$${\rm rk}(B)={\rm rk}(B(G_{k}B_{k})^{-k-1}GG^{*})\leq {\rm rk}(B(G_{k}B_{k})^{-k-1}G)\leq rk(B).$$
Thus, we have the following equalities:
$$R(X)=R(B(BG_{k}B_{k})^{\dagger})=R(B(BG_{k}B_{k})^{*})=R(B),$$
$$R(X^{*})=R(B(B^{*}B)^{-1}((G_{k}B_{k})^{-1})^{*}B^{*})=R(B),$$
$$R(A^{D})=R(B(G_{k}B_{k})^{-k-1}G)=R(B).$$
Namely, $$R(X)=R(X^{*})=R(A^{D}).$$
Similarly, we can verify $A_{\tiny\textcircled{\tiny$\dagger$}}=(G_{k}B_{k}G)^{\dagger}G.$ This completes the proof.
\end{proof}

Let $k=1$ in Theorem \ref{e2}. Then we obtain the following corollary.
\begin{corollary} \label{e4} Let $A\in\mathbb{C}^{n\times n}$ with ${\rm ind}(A)=1.$ If $A=BG$ is a full-rank decomposition,
then
\begin{eqnarray*}
  A^{\tiny\textcircled{\tiny\#}} &=& \lim_{\lambda\rightarrow 0}B(AB)^{*}(AB(AB)^{*}+\lambda I)^{-1} \\
   &=&  \lim_{\lambda\rightarrow 0}B((AB)^{*}AB+\lambda I)^{-1}(AB)^{*}
\end{eqnarray*}
and
\begin{eqnarray*}
  A_{\tiny\textcircled{\tiny\#}}   &=&   \lim_{\lambda\rightarrow 0}((GA)^{*}GA+\lambda I)^{-1}(GA)^{*}G\\
   &=&  \lim_{\lambda\rightarrow 0}(GA)^{*}(GA(GA)^{*}+\lambda I)^{-1}G.
 \end{eqnarray*}
\end{corollary}

\section{The second and third approaches }
In this section, we present two types of limit representations of the core-EP and dual core-EP inverse, which depend on their own explicit representation. In particular, limit representations of the core and dual core inverse are also given.

Based on Definition \ref{cp} and Lemma \ref{db4}, we present the second approach by using the following equation firstly:
$$A^{\tiny\textcircled{\tiny$\dagger$}}A^{k+1}=A^{D}A^{k}(A^{k})^{\dagger}A^{k+1}=A^{k},$$
where $k={\rm ind}(A).$

\begin{theorem}\label{f1} Let $A\in\mathbb{C}^{n\times n}$ with ${\rm ind}(A)=k.$ Then
\begin{itemize}
\item[$(1)$] $A^{\tiny\textcircled{\tiny$\dagger$}}=\lim\limits_{\lambda\rightarrow 0}A^{k}(A^{k})^{\ast}(A^{k+1}(A^{k})^{\ast}+\lambda I)^{-1},$
\item[$(2)$] $A_{\tiny\textcircled{\tiny$\dagger$}}=\lim\limits_{\lambda\rightarrow 0}((A^{k})^{\ast}A^{k+1}+\lambda I)^{-1}(A^{k})^{\ast}A^{k}.$
\end{itemize}
\end{theorem}

\begin{proof} Suppose that $A\in\mathbb{C}^{n\times n}$ with ${\rm ind}(A)=k.$ Let the core-EP decomposition be as in Lemma \ref{db3}. From \cite{HXW}, we know that there exists a unitary matrix $U$ such that $A=U\left[ {{\begin{matrix}
  T&S\\
  0&N\\
\end{matrix} }} \right]U^{*},$ where $T$ is a nonsingular matrix, ${\rm rk}(T)={\rm rk}(A^{k})$ and $N$ is a nilpotent matrix of index $k.$ The core-EP inverse of $A$ is $A^{\tiny\textcircled{\tiny$\dagger$}}=U\left[ {{\begin{matrix}
  T^{-1}&0\\
  0&0\\
\end{matrix} }} \right]U^{*}.$ By a direct computation, we obtain
\begin{equation*}\label{l1}
  A^{k}=U\left[ {{\begin{matrix}
  T^{k}&\hat{T}\\
  0&0\\
\end{matrix} }} \right]U^{*}, A^{k+1}=U\left[ {{\begin{matrix}
  T^{k+1}&T\hat{T}\\
  0&0\\
\end{matrix} }} \right]U^{*},
\end{equation*}
\begin{equation}\label{20}
  (A^{k})^{*}=U\left[ {{\begin{matrix}
  (T^{k})^{\ast}&0\\
  \hat{T}^{*}&0\\
\end{matrix} }} \right]U^{\ast},~where ~\hat{T}=\sum\limits_{i=0}^{k-1}T^{i}SN^{k-1-i}.
\end{equation}

\begin{equation}\label{l1}
  A^{k}(A^{k})^{*}=U\left[ {{\begin{matrix}
  T^{k}(T^{k})^{\ast}+\hat{T}(\hat{T})^{*}&0\\
  0&0\\
\end{matrix} }} \right]U^{\ast},
\end{equation}
\begin{equation}\label{l2}
  A^{k+1}(A^{k})^{\ast}+\lambda I_{n}=U\left[ {{\begin{matrix}
  T^{k+1}(T^{k})^{\ast}+T\hat{T}(\hat{T})^{*}+\lambda I_{{\rm rk}(T)}&0\\
  0&\lambda I_{n-{\rm rk}(T)}\\
\end{matrix} }} \right]U^{\ast}.
\end{equation}
Since $T$ is nonsingular, $T^{k}(T^{k})^{*}+\hat{T}(\hat{T})^{*}$ is positive definite matrix.
Combining with (\ref{l1})and (\ref{l2}), we have

\begin{eqnarray*}
\begin{split}
&\lim_{\lambda\rightarrow 0}A^{k}(A^{k})^{\ast}(A^{k+1}(A^{k})^{\ast}+\lambda I)^{-1}
=\lim_{\lambda\rightarrow 0}U\left[{{\begin{matrix}
 T^{k}(T^{k})^{\ast}+\hat{T}(\hat{T})^{*}&0\\
  0&0\\
\end{matrix}}}\right]\times\\
 &\left[{{\begin{matrix}
  T^{k+1}(T^{k})^{\ast}+T\hat{T}(\hat{T})^{*}+\lambda I_{{\rm rk}(T)}&0\\
 0&\lambda I_{n-{\rm rk}(T)}
\end{matrix}}}\right]^{-1} U^{*}
=U\left[ {{\begin{matrix}
  T^{-1}&0\\
  0&0\\
\end{matrix} }} \right]U^{*}.
\end{split}
\end{eqnarray*}
(2) It is analogous.
\end{proof}

Let $k=1,$ we have the following corollary.

\begin{corollary}\label{f2} Let $A\in\mathbb{C}^{n\times n}$ with ${\rm ind}(A)=1.$ Then
\begin{itemize}
\item[$(1)$] $A^{\tiny\textcircled{\tiny\#}}=\lim\limits_{\lambda\rightarrow 0}AA^{\ast}(A^{2}A^{\ast}+\lambda I)^{-1},$
\item[$(2)$] $A_{\tiny\textcircled{\tiny\#}}=\lim\limits_{\lambda\rightarrow 0}(A^{\ast}A^{2}+\lambda I)^{-1}A^{\ast}A.$
\end{itemize}
\end{corollary}

Next, we present the third approach of limit representations of the core-EP inverse. From Definition \ref{cp} and equation (\ref{c16}), it is easy to know that
$$A^{\tiny\textcircled{\tiny$\dagger$}}=A^{k}(A^{k+1})^{\dagger} ~~and~~A_{\tiny\textcircled{\tiny$\dagger$}}=(A^{k+1})^{\dagger}A^{k}, $$
where $k={\rm ind}(A).$
Combining with limit representation of the Moore-Penrose inverse of $A,$ we have the following theorem.

\begin{theorem}\label{f1} Let $A\in\mathbb{C}^{n\times n}$ with ${\rm ind}(A)=k.$ Then
\begin{eqnarray*}
  A^{\tiny\textcircled{\tiny$\dagger$}}&=&\lim_{\lambda\rightarrow 0}A^{k}(A^{k+1})^{*}(A^{k+1}(A^{k+1})^{*}+\lambda I)^{-1}  \\
   &=&\lim_{\lambda\rightarrow 0}A^{k}((A^{k+1})^{*}A^{k+1}+\lambda I)^{-1}(A^{k+1})^{*}
\end{eqnarray*}
and
\begin{eqnarray*}
  A_{\tiny\textcircled{\tiny$\dagger$}}&=&\lim_{\lambda\rightarrow 0}((A^{k+1})^{*}A^{k+1}+\lambda I)^{-1}(A^{k+1})^{*}A^{k}\\
   &=&\lim_{\lambda\rightarrow 0}(A^{k+1})^{*}(A^{k+1}(A^{k+1})^{*}+\lambda I)^{-1}A^{k}
\end{eqnarray*}
\end{theorem}

Let $k=1,$ we have the following corollary.

\begin{corollary}\label{f2} Let $A\in\mathbb{C}^{n\times n}$ with ${\rm ind}(A)=1.$ Then
\begin{eqnarray*}
   A^{\tiny\textcircled{\tiny\#}} &=&\lim_{\lambda\rightarrow 0}A(A^{2})^{*}(A^{2}(A^{2})^{*}+\lambda I)^{-1}  \\
   &=&  \lim_{\lambda\rightarrow 0}A((A^{2})^{*}A^{2}+\lambda I)^{-1}(A^{2})^{*}
\end{eqnarray*}
and
\begin{eqnarray*}
   A_{\tiny\textcircled{\tiny\#}}  &=& \lim_{\lambda\rightarrow 0}((A^{2})^{*}A^{2}+\lambda I)^{-1}(A^{2})^{*}A \\
   &=& \lim_{\lambda\rightarrow 0}(A^{2})^{*}(A^{2}(A^{2})^{*}+\lambda I)^{-1}A.
\end{eqnarray*}
\end{corollary}

Let $k=1.$ From \cite{BT}, we know that $A^{\tiny\textcircled{\tiny\#}}=(A^{2}A^{\dagger})^{\dagger}.$ According to the above corollary, it is known that
$A^{\tiny\textcircled{\tiny\#}}=A(A^{2})^{\dagger}.$ Since the core inverse is unique, we obtain the following corollary.

\begin{corollary}\label{f3} Let $A\in\mathbb{C}^{n\times n}$ with ${\rm ind}(A)=1.$ Then $(A^{2}A^{\dagger})^{\dagger} =A(A^{2})^{\dagger}.$
\end{corollary}

\section{Examples}
In this section, we present two examples to illustrate the efficacy of the established limit representations in this paper.

\begin{example}\label{Ex5-1}
Let
$A=\left[ {{\begin{matrix}
1&1&2&5  \\
0&1&1&1  \\
0&3&3&1    \\
 1&0&1&4   \\
\end{matrix} }} \right]$, where ${\rm ind}(A)=2$ and ${\rm rk}(A)=3.$
Let $B_{1}=\left[ {{\begin{matrix}
1&1&5 \\
 0&1&1\\
 0&3&1  \\
1&0&4  \\
\end{matrix} }} \right],$  $G_{1}=\left[ {{\begin{matrix}
1&0&1&0\\
0&1&1&0\\
0&0&0&1\\
\end{matrix} }} \right],$
$B_{2}=\left[ {{\begin{matrix}
1&1\\
1&0\\
0&1\\
\end{matrix} }} \right],$  $G_{2}=\left[ {{\begin{matrix}
0&4&2\\
1&0&4\\
\end{matrix} }} \right].$\\
The exact core-EP inverse of $A$ is equal to\\
$A^{\tiny\textcircled{\tiny$\dagger$}}=\frac{1}{756}\left[ {{\begin{matrix}
80&4&-24&76\\
8&13&48&-5\\
-8&50&204&-58\\
72&-9&-72&81\\
\end{matrix} }} \right].$
Set $B=B_{1}B_{2}$ and $E=G_{2}B_{2}.$\\ Using matlab, we have
$$\lim_{\lambda\rightarrow 0}B(BE)^{*}(BE(BE)^{*}+\lambda I)^{-1}=\left[ {{\begin{matrix}
 \frac{20}{189}&\frac{1}{189}&-\frac{2}{63}&\frac{19}{189}\\
 \frac{2}{189}&\frac{13}{756}&\frac{4}{63}& -\frac{5}{756}\\
-\frac{2}{189}&\frac{25}{378}&\frac{17}{63}&-\frac{29}{378}\\
 \frac{2}{21}&-\frac{1}{84}&-\frac{2}{21}&\frac{3}{28}\\
\end{matrix} }} \right],$$

$$\lim_{\lambda\rightarrow 0}A^{2}(A^{2})^{*}(A^{3}(A^{2})^{*}+\lambda I)^{-1}=\left[ {{\begin{matrix}
 \frac{20}{189}&\frac{1}{189}&-\frac{2}{63}&\frac{19}{189}\\
 \frac{2}{189}&\frac{13}{756}&\frac{4}{63}& -\frac{5}{756}\\
-\frac{2}{189}&\frac{25}{378}&\frac{17}{63}&-\frac{29}{378}\\
 \frac{2}{21}&-\frac{1}{84}&-\frac{2}{21}&\frac{3}{28}\\
\end{matrix} }} \right],$$

$$\lim_{\lambda\rightarrow 0}A^{2}(A^{3})^{*}(A^{3}(A^{3})^{*}+\lambda I)^{-1}=\left[ {{\begin{matrix}
\frac{20}{189}&\frac{1}{189}&-\frac{2}{63}&\frac{19}{189}\\
 \frac{2}{189}&\frac{13}{756}&\frac{4}{63}& -\frac{5}{756}\\
-\frac{2}{189}&\frac{25}{378}&\frac{17}{63}&-\frac{29}{378}\\
 \frac{2}{21}&-\frac{1}{84}&-\frac{2}{21}&\frac{3}{28}\\
\end{matrix} }} \right].$$
\end{example}

\begin{example}\label{Ex5-2} Let $A=\left[ {{\begin{matrix}
1&0&3\\
4&0&2\\
2&0&1\\
\end{matrix} }} \right],$ where ${\rm rk}(A)=2$ and ${\rm ind}(A)=1.$ Let \\
$B=\left[ {{\begin{matrix}
1&3\\
4&2\\
2&1\\
\end{matrix} }} \right],$$G=\left[ {{\begin{matrix}
1&0&0\\
0&0&1\\
\end{matrix} }} \right].$ According to Lemma \ref{db2}, the exact core inverse of $A$ is equal to\\
$$A^{\tiny\textcircled{\tiny\#}}=B(B^{*}BGB)^{-1}B^{*}=\left[ {{\begin{matrix}
-0.2000&0.2400&0.1200\\
0.8000&-0.1600&-0.0800\\
0.4000 &-0.0800 &-0.0400\\
\end{matrix} }} \right]. $$
Using matlab, we obtain
$$\lim_{\lambda\rightarrow 0}B(AB)^{*}(AB(AB)^{*}+\lambda I)^{-1}=\left[ {{\begin{matrix}
-\frac{1}{5}&\frac{6}{25}&\frac{3}{25}\\
\frac{4}{5}&-\frac{4}{25}&-\frac{2}{25}\\
\frac{2}{5}&-\frac{2}{25}&-\frac{1}{25}\\
\end{matrix} }} \right],$$

$$\lim_{\lambda\rightarrow 0}AA^{*}(A^{2}A^{*}+\lambda I)^{-1}=\left[ {{\begin{matrix}
-\frac{1}{5}&\frac{6}{25}&\frac{3}{25}\\
\frac{4}{5}&-\frac{4}{25}&-\frac{2}{25}\\
\frac{2}{5}&-\frac{2}{25}&-\frac{1}{25}\\
\end{matrix} }} \right],$$

$$\lim_{\lambda\rightarrow 0}A(A^{2})^{*}(A^{2}(A^{2})^{*}+\lambda I)^{-1}=\left[ {{\begin{matrix}
-\frac{1}{5}&\frac{6}{25}&\frac{3}{25}\\
\frac{4}{5}&-\frac{4}{25}&-\frac{2}{25}\\
\frac{2}{5}&-\frac{2}{25}&-\frac{1}{25}\\
\end{matrix} }} \right].$$
\end{example}

\section*{ Acknowledgements}
This research is supported by the National Natural Science Foundation of China (No.11771076, No.11471186).

\end{CJK*}

\section*{References}
\begin{thebibliography}{99}
\bibitem{BT} O.M. Baksalary, G. Trenkler, Core inverse of matrices, Linear Multilinear Algebra 58 (2010) 681-697.
\bibitem{BG} A. Ben-Israel, T.N.E. Greville, Generalized Inverses: Theory and Applications, second ed., Springer-Verlag, New York, 2003.
\bibitem{CM} S.L. Campbell, C.D. Meyer Jr., Generalized Inverse of Linear Transformation, Pitman London, 1979, Dover, New York, 1991.
\bibitem{ZPC} J.L. Chen, H.H. Zhu, P. Patri\'{c}io, Y.L. Zhang, Characterizations and representations of core and dual core inverses, Canad. Math. Bull. 60 (2017) 269-282.
\bibitem{YCH} Y.L. Chen, The generalized Bott-Duffin inverse and its applications, Linear Algebra Appl. 134 (1990) 71-91.
\bibitem{YC} Y.F. Gao, J.L. Chen, Pseudo core inverses in rings with involution, Comm. Algebra 46 (2018) 38-50.
\bibitem{J} J. Ji, An alternative limit expression of Drazin inverse and its application, Appl. Math. Comput. 61 (1994) 151-156.
\bibitem{KR} R.E. Kalaba, N. Rasakhoo, Algorithm for generalized inverse. J. Optimization Theory and Appl. 48 (1986) 427-435.
\bibitem{LC} T.T. Li, J.L. Chen, Characterizations of core and dual core inverses in rings with involution, Linear Multilinear Algebra Doi:10.108003081087.2017.1320963.
\bibitem{LYZW} X.J. Liu, Y.M. Yu, J. Zhong, Y.M. Wei, Integral and limit representations of the outer inverse in Banach space, Linear Multilinear Algebra 60 (2012) 333-347.
\bibitem{MPM} K. Manjunatha Prasad, K.S. Mohana, Core-EP inverse, Linear Multilinear Algebra 62 (2014) 792-802.
\bibitem{M} C.D. Meyer, Limits and the index of a square matrix, SIAM J. Appl. Math. 26 (1974) 469-478.
\bibitem{RDD} D.S. Raki\'{c}, N. \v{C}. Din\v{c}i\'{c}, D.S. Djordjevi\'{c},  Group, Moore-Penrose, core and dual core inverse in rings with involution, Linear Algebra Appl. 463 (2014) 115-133.
\bibitem{CSK} C.R. Rao, S.K. Mitra, Generalized Inverse of Matrices and its Applications, Wiley, New York, 1971.
\bibitem{S} P.S. Stanimirovi\'{c}, Limit representations of generalized inverses and related methods, Appl. Math. Comput. 103 (1999) 51-68.
\bibitem{HXW} H.X. Wang, Core-EP decomposition and its applications, Linear Algebra Appl. 508 (2016) 289-300.
\bibitem{WL} H.X. Wang, X.J. Liu, Characterizations of the core inverse and the core partial ordering, Linear Multilinear Algebra 63 (2015) 1829-1836.
\bibitem{W} Y.M. Wei, A characterization and representation of the generalized inverse $A^{(2)}_{T,S}$ and its applications, Linear Algebra Appl. 280 (1998) 87-96.
\bibitem{XCZ} S.Z. Xu, J.L. Chen, X.X. Zhang, New characterizations for core inverses in rings with involution, Front. Math. China 12 (2017) 231-246.
\end{thebibliography}
\end{document}